\documentclass[reqno]{amsart}
\usepackage{hyperref}

\AtBeginDocument{{\noindent\small
\emph{}}
\vspace{8mm}}

\begin{document}
\title[]
{On a class of quasilinear elliptic equation with indefinite weights on graphs}

\author[S. Man]{Shoudong Man}
\address{Shoudong Man \\ Department of Mathematics \\ Tianjin University of Finance and Economics \\ Tianjin 300222, P. R. China}
\email{manshoudong@163.com; shoudongmantj@tjufe.edu.cn}

\author[G. Zhang]{Guoqing Zhang}
\address{Guoqing Zhang \\ College of Sciences \\ University of Shanghai for Science and Technology \\ Shanghai, 200093, P. R.  China}
\email{shzhangguoqing@126.com}
\thanks{The first author is supported by the National Natural Science Foundation of China (Grant No. 11601368)}
\thanks{The second author is supported by the National Natural Science Foundation of China (Grant No.11771291)}

\dedicatory{}


\begin{abstract}
~~Suppose that $G=(V, E)$ is a connected locally finite graph with the vertex set $V$ and the edge set $E$. Let $\Omega\subset V$  be a bounded domain.
Consider the following quasilinear elliptic equation
on graph $G$
 $$
\left \{
\begin{array}{lcr}
 -\Delta_{p}u= \lambda K(x)|u|^{p-2}u+
  f(x,u), \ \  x\in\Omega^{\circ},\\
  u=0, \ \ x\in\partial \Omega,
 \end{array}
\right.
$$
where $\Omega^{\circ}$ and $\partial \Omega$
denote the interior and the boundary of $\Omega$ respectively, $\Delta_{p}$ is the discrete $p$-Laplacian, $K(x)$ is a given function
 which may change sign, $\lambda$ is the eigenvalue parameter and $f(x,u)$ has exponential growth.
We prove the existence and monotonicity of the principal eigenvalue of the corresponding eigenvalue problem.
Furthermore,  we also obtain the existence of a positive solution by using variational methods.
\end{abstract}

\maketitle
\numberwithin{equation}{section}
\newtheorem{theorem}{Theorem}[section]
\newtheorem{lemma}[theorem]{Lemma}
\newtheorem{remark}[theorem]{Remark}
\allowdisplaybreaks

\section{Introduction}

In this paper, we consider the following quasilinear elliptic equation with indefinite weights on graph $G$
 $$
\left \{
\begin{array}{lcr}
 -\Delta_{p}u= \lambda K(x)|u|^{p-2}u+
  f(x,u), \ \  x\in\Omega^{\circ},\\
  u=0, \ \ x\in\partial \Omega,
 \end{array}
\right.\eqno{(1.1)}
$$
where $G=(V, E)$ is a locally finite graph, $\Omega\subset V$ is a bounded domain,
$\Omega^{\circ}$ and $\partial \Omega$
denote the interior and the boundary of $\Omega$ respectively, $\Delta_{p}$ denotes the discrete $p$-Laplacian, $\lambda$ is the eigenvalue parameter and
 $K(x)$ is a given function which satisfies
 $$K^{+}(x)\not\equiv 0,~ K\in L^{1}(\Omega),~K^{\pm}(x)=\max\{\pm K(x),0\}.\eqno{(1.2)}$$

 Quasilinear elliptic equations have been studied extensively on Euclidean domain and Riemannian manifold.
 Zhang and Liu \cite{liu} investigated the following critical elliptic equations with
indefinite weights
$$\left \{
\begin{array}{lcr}
-\Delta u-\mu \frac{u}{(|x|\ln\frac{R}{|x|})^{2}}= \lambda K(x)u+
  f(x,u),   x \in \Omega,\\
  u=0,  x\in\partial \Omega,
\end{array}
\right.\eqno{(1.3)}
$$
where $K(x)$ satisfies (1.2), and proved the existence of a nontrivial
solution by using the Mountain Pass Lemma. As for the $p$-Laplacian, in particular, Fan and Li \cite{Fan} discussed
$$\left \{
\begin{array}{lcr}
-\triangle_{p}u=\lambda |u|^{p-2}u+f(x,u),
x\in\Omega,\\
u=0,  x\in\partial\Omega,
\end{array}
\right.\eqno{(1.4)}
$$
where $0<\lambda<\lambda_{2}$  and $\lambda_{2}$ is the second
eigenvalue of the $p$-Laplacian with indefinite weights, and obtained the
existence of a nontrivial solution for the problem (1.4).
In \cite{Xuan}, B. Xuan studied the following elliptic equation
$$\left \{
\begin{array}{lcr}
-\triangle_{p}u=\lambda V|u|^{p-2}u+f(x,u),
x\in\Omega,\\
u=0,  x\in\partial\Omega,
\end{array}
\right.\eqno{(1.5)}
$$
where $p>1$, $\Omega\subset R^{N}$ is a bounded domain, and $V(x)$ is a given function satisfying
$$V^{+} \not\equiv 0 \ \ and \ \ V\in L^{s}(\Omega). $$
B. Xuan obtained the existence of a nontrivial weak solution for the problem (1.5) in the case of
$0 <\lambda< \lambda_{1}$ by the Mountain Pass Lemma and in the case of $\lambda_{1} \leq \lambda < \lambda_{2}$ by the Linking
Argument Theorem respectively.
In \cite{Lancelott}, M. Degiovanni and S. Lancelotti studied the problem (1.5)
when $V\in L^{\infty}(\Omega)$ and $f(x,u)$ is subcritical and
superlinear at $0$ and at infinity respectively, and they proved a nontrivial
solution for the problem (1.5) for any $\lambda\in \mathbb{R}$ .

Most recently, the investigation of discrete weighted Laplacians and various equations on graphs have attracted much attention
 [4-9, 13-15].
 A. Grigor'yan, Y. Lin and Y. Y. Yang \cite{AGY} studied the following Yamabe type equation
  $$
\left \{
\begin{array}{lcr}
 -\Delta_{p}u+ h(x)|u|^{p-2}u= f(x,u),  \ \  x\in\Omega^{\circ},\\
  u\geq0,  \ \  x\in\Omega^{\circ}, u=0 \ \ x\in\partial \Omega,
 \end{array}
\right. \eqno{(1.6)}
$$
 and obtained the existence of a positive solution if $h(x)>0$.
Ge \cite{HUABIN} studied the following $p$-th Yamabe equation
$$ -\Delta_{p}u+ h(x)u^{p-1}= \lambda g u^{\alpha-1},  \ \  x\in\Omega^{\circ}, \eqno{(1.7)}$$
where $\alpha\geq p>1$, and proved the existence of a positive solution.
Zhang and Lin \cite{XA} proved the problem (1.7) has at least a positive solution as $2<\alpha\leq p$.

In this paper, we study the quasilinear elliptic equation (1.1) on graph $G$,
which can be viewed as a discrete version of the equation (1.5) studied by B. Xuan, etc.
Firstly, we consider the following eigenvalue problem with indefinite weights on graph $G$
 $$
\left \{
\begin{array}{lcr}
 -\Delta_{p}u= \lambda K(x)|u|^{p-2}u, \ \ x\in\Omega^{\circ},\\
 u=0, \ \ x\in\partial \Omega,
 \end{array}
\right. \eqno{(1.8)}
$$
where $K(x)$ satisfies the assumption (1.2), and obtain the existence and monotonicity of the principal eigenvalue.
Secondly, using  Mountain Pass Lemma and  Sobolev embedding theorem on graph $G$,
 we  obtain the existence of a positive solution of the problem (1.1) as the nonlinear term $f(x,s)$ has exponential growth as $s\rightarrow +\infty$.

This paper is organized as follows. In Section 2, we introduce some notations and lemmas on graph $G$ , and state our main results.
 In Section 3, we prove our main theorems.

\section{Preliminaries and main results}
Let $G = (V, E)$ be a graph. The degree of vertex $x$, denoted by $\mu(x)$, is the number of edges connected to $x$.
 If $\mu(x)$ is finite for every vertex $x$ of $V$, we say that $G$ is a locally finite graph.
We denote $x\sim y$ if vertex $x$ is adjacent to vertex $y$, and $\omega_{xy}=\omega_{yx}>0$ is the edge weight.
The finite measure $\mu(x)=\sum_{y\sim x}\omega_{xy}$. The boundary of $\Omega$ is defined as
$\partial\Omega=\{y\bar{\in}\Omega:\exists x\in \Omega ~such ~that ~xy\in E \}$ and the interior of $\Omega$ is denoted
by $\Omega^{\circ}=\Omega\setminus \partial\Omega$.
A graph $G$ is called connected if for any vertices $x,y\in V$, there exists $\{x_{i}\}_{i=0}^{n}$ that satisfies
$x=x_{0}\sim x_{1}\sim x_{2}\sim \cdot\cdot\cdot \sim x_{n}=y$. In this paper, we suppose that all graphs are connected.

  From \cite{AGY}, for any function $u:\Omega\rightarrow \mathbb{R}$, the $\mu$-Laplacian of $u$ is defined as
$$\Delta u(x)=\frac{1}{\mu(x)}\sum_{y\sim x}\omega_{xy}[u(y)-u(x)].\eqno{(2.1)}$$
The associated gradient form reads
$$\Gamma(u,v)(x)=\frac{1}{2}\{\Delta(u(x)v(x))-u(x)\Delta v(x)-v(x)\Delta u(x)\} \eqno{(2.2)}$$
$$~~~~~~~~~~~=\frac{1}{2\mu(x)}\sum_{y\sim x} \omega_{xy}(u(y)-u(x))(v(y)-v(x)).\eqno{(2.3)}$$
The length of the gradient for $u$ is
$$|\nabla u|(x)=\sqrt{\Gamma(u,u)(x)}=(\frac{1}{2\mu(x)}\sum_{y\sim x} \omega_{xy}(u(y)-u(x))^{2})^{1/2}.\eqno{(2.4)} $$
For any function $u: \Omega\rightarrow \mathbb{R}$, we denote
$$\int_{\Omega}ud\mu=\sum_{x\in\Omega}\mu(x)u(x),\eqno{(2.5)}$$
and set
$$Vol(G)=\int_{\Omega}d\mu.\eqno{(2.6)}$$
The $p$-Laplacian of $u:\Omega\rightarrow \mathbb{R}$, namely $\Delta_{p}u$, is defined in the distributional sense as
$$\int_{\Omega}(\Delta_{p}u)\phi d\mu=-\int_{\Omega}|\nabla u|^{p-2}\Gamma(u,\phi)d\mu, ~~~~\forall \phi\in C_{c}(\Omega), \eqno{(2.7)}$$
where $C_{c}(\Omega)$ is the set of all functions with compact support. So, $\Delta_{p}u$ can be written as
$$\Delta_{p} u(x)=\frac{1}{2\mu(x)}\sum_{y\sim x}\omega_{xy}(|\nabla u|^{p-2}(y)+|\nabla u|^{p-2}(x))(u(y)-u(x)).\eqno{(2.8)}$$

For any $p>1$, $W^{1,p}(\Omega)$ is defined as a space of all functions $u: \Omega\rightarrow \mathbb{R}$ satisfying
$$||u||_{W^{1,p}(\Omega)}=(\int_{\Omega}|\nabla u|^{p}d\mu+\int_{\Omega}| u|^{p}d\mu)^{1/p}<\infty. \eqno{(2.9)}$$
Denote $C_{0}^{1}(\Omega)$ as a set of all functions $u: \Omega\rightarrow \mathbb{R}$ with $u=0$ on $\partial\Omega$,
and $W_{0}^{1,p}(\Omega)$ as
the completion of $C_{0}^{1}(\Omega)$  under the norm (2.9).

\begin{lemma}(\cite{AGY} Theorem $7$)   \label{y2}
Let $G =(V,E)$ be a locally finite graph and $\Omega$ be a bounded domain of  $V$ such that $\Omega^{0}\neq \emptyset$. For any $p > 1$,
 $W_{0}^{1,p}(\Omega)$ is embedded in $ L^{q}(\Omega)$ for all
$1\leq  q \leq +\infty$. In particular, there exists a constant $C$ depending only on $p$ and $\Omega$
 such that
$$(\int_{\Omega}|u|^{q}d\mu)^{1/q}\leq C(\int_{\Omega}|\nabla u|^{p}d\mu)^{1/p}, \eqno{(2.10)}$$
for all $1\leq q \leq +\infty$ and  $u\in W_{0}^{1,p}(\Omega)$. Moreover, $W_{0}^{1,p}(\Omega)$ is pre-compact, namely, if $\{u_{k}\}$ is
bounded in $W_{0}^{1,p}(\Omega)$, then up to a subsequence, there exists some $u\in W_{0}^{1,p}(\Omega)$
 such that $u_{k}\rightarrow u$ in $ W_{0}^{1,p}(\Omega)$.
\end{lemma}
By Lemma 2.1, we obtain that $W_{0}^{1,p}(\Omega)$ is a Banach space.
\begin{lemma}(\cite{YLL}Mountain pass lemma) \label{y1}
Let $(X, ||\cdot ||)$ be a Banach space, $J \in C^{1}(X,\mathbb{R})$, $e \in X$
and $r > 0$ such that $||e|| > r$ and $b=\inf_{||u||=r}J(u)>J(0)>J(e)$.
If J satisfies the $(PS)_{c}$ condition with
 $c=\inf_{\gamma\in\Gamma}\max_{t\in[0,1]}J(\gamma(t))$,
where $\Gamma=\{\gamma\in C([0,1], X): \gamma(0)=0,\gamma(1)=e\}$,
then c is a critical value of J.
\end{lemma}

  For the well-know Trudinger-Moser inequality on Euclidean domain and complete Riemannian manifold \cite{MOSER, tang}, by Lemma \ref{y2}, we have
\begin{lemma}\label{moser}(Trudinger-Moser inequality on locally finite graphs)
Suppose that $G=(V,E)$ is a locally finite graph. Let $\Omega\subset V$  be a bounded domain.
Then there exists a
constant $C$ which depends only on $p$ and $\Omega$ such that
$$
\sup_{||u||_{W_{0}^{1,p}(\Omega)}\leq 1}\int_{\Omega}exp(\alpha|u|^{\frac{p}{p-1}})d\mu\leq C|\Omega|\ \ for\ any \ \ \alpha>1 \ and \ p>2, \eqno{(2.11)}
$$
where $|\Omega|=\int_{\Omega}d\mu(x)=$ Vol $\Omega$,
and Vol $\Omega$ denotes the volume of the subgraph $\Omega$.
\end{lemma}
\begin{proof}
For any function $u$ which satisfies $||u||_{W_{0}^{1,p}(\Omega)}\leq 1$, by Lemma \ref{y2} and $\frac{p}{p-1}>1$, we obtain
that there exists a constant $C_{0}$ such that
$$
(\int_{\Omega}|u|^{\frac{p}{p-1}}d\mu )^{\frac{p-1}{p}}
\leq C_{0}(\int_{\Omega}|\nabla u|^{p}d\mu )^{\frac{1}{p}}
=C_{0}||u||_{W_{0}^{1,p}(\Omega)}
\leq C_{0}. \eqno{(2.12)}
$$
 Denote $\mu_{min}= min_{x\in \Omega}\mu(x)$. Then (2.12) leads to
$$
||u||_{L^{\infty}(\Omega)}\leq \frac{C_{0}}{\mu_{min}}.\eqno{(2.13)}
$$
Thus for any $\alpha>1$ and $p>2$, we have
$$
(\int_{\Omega}exp(\alpha|u|^{\frac{p}{p-1}})d\mu )^{\frac{p-1}{p}}
\leq exp(\frac{\alpha C_{0}}{\mu_{min}})|\Omega|^{\frac{p-1}{p}}.\eqno{(2.14)}
$$
So, we have
$$
\sup_{||u||_{W_{0}^{1,p}(\Omega)}\leq 1}\int_{\Omega}exp(\alpha|u|^{\frac{p}{p-1}})d\mu\leq C|\Omega|,\eqno{(2.15)}
$$
where $C=(exp(\frac{\alpha C_{0}}{\mu_{min}}))^{\frac{p}{p-1}}$.
\end{proof}

In order to find the principal eigenvalue of the problem (1.8), we solve the following minimization problem
$$(P) \ \ minimize \int_{\Omega}|\nabla u|^{p}d\mu,\ \ u\in W_{0}^{1,p}(\Omega) \ and \ \int_{\Omega}K(x)|u|^{p}d\mu=1.  \eqno{(2.16)}$$

Now, we state our main theorems.

 \begin{theorem}\label{T11}
 Under the assumption $(1.2)$, Problem (P) has a solution $e_{1}\geq 0$. Moreover, $e_{1}$ is an eigenvalue of the problem (1.8)
 corresponding to the principal eigenvalue $\lambda_{1}=\int_{\Omega}|\nabla e_{1}|^{p}d\mu$.
\end{theorem}

\begin{theorem}\label{T1}
 Let $K_{1}(x)$ and $K_{2}(x)$ be two
weights which satisfy the assumption (1.2). Assume $K_{1}(x)<K_{2}(x)$ for all $x\in \Omega$ and
$\{{x\in\Omega}:K_{1}<K_{2}\}\neq \emptyset$. Then $\lambda_{1}(K_{2})<\lambda_{1}(K_{1})$.
\end{theorem}

\begin{theorem}\label{T2} Let $\Omega_{1}$ be a proper bounded open
subset of a bounded domain $\Omega_{2}\subset K$. Then
$\lambda_{1}(\Omega_{2})\leq\lambda_{1}(\Omega_{1})$.
\end{theorem}

\begin{theorem}\label{T3}
   Let $G = (V, E)$ be a locally finite graph and $\Omega\subset V$ be a bounded domain with $\Omega^{0}\neq \emptyset$.
Let $\lambda_{1}$ be defined as in Theorem $\ref{T11}$. Set $0<\lambda<\lambda_{1}$ and $p>2$. Suppose that
$f: \Omega\times \mathbb{R}\rightarrow \mathbb{R}$ satisfies the following hypotheses:\\
\noindent$(H1)$ For any $x\in \Omega$,  $f(x,t)$ is continuous in $t\in \mathbb{R}$;\\
\noindent$(H2)$ For all $(x,t)\in \Omega\times[0,+\infty), f(x,t)\geq 0$, and $f(x,0)=0$ for all $x\in \Omega$;\\
\noindent$(H3)$ $f(x,t)$ has exponential growth at $+\infty$, that is, for all $\alpha>1$,
$$\lim_{t\rightarrow +\infty}\frac{f(x,t)}{exp(\alpha
|t|^{\frac{p}{p-1}})}=0;\ \eqno{(2.17)}$$
\noindent$(H4)$ For any $x\in \Omega$, there holds $\lim_{t\rightarrow 0+}\frac{f(x,t)}{t^{p-1}}=0$;\\
\noindent$(H5)$ There exists $q>p>2$ and $s_{0}>0$ such that if $s\geq s_{0}$, then there holds\\
~~~\indent $0<qF(x,s)<f(x,s)s$ for any $x\in \Omega$, where $F(x,s)=\int_{0}^{s}f(x,t)dt$.

\noindent Then there exists a positive solution for the problem (1.1).
\end{theorem}

\section{The proof of  main Theorems}

 \indent{ \textbf{The Proof of Theorem 2.4.}}
Now, we introduce the functional $\phi, \psi: W_{0}^{1,p}(\Omega)\rightarrow \mathbb{R}$ defined by
$$\phi(u)=\int_{\Omega}|\nabla u|^{p}d\mu, \ and \   \psi(u)=\int_{\Omega}K(x)|u|^{p}d\mu. \eqno{(3.1)}$$
Let us also introduce the set
$$ M=\{u\in W_{0}^{1,p}(\Omega):  \psi(u)=1\}.\eqno{(3.2)}$$
By the assumption (1.2), we have $M\neq \emptyset$ and $M$ is a manifold of $C^{1}$ in $ W_{0}^{1,p}(\Omega)$.
Obviously the functional $\phi(u) $ is bounded from below. Hence,
let $\{u_{n}\}$ be a minimizing sequence for the problem ($P$)
such that
$$\frac{1}{p}\int_{\Omega}|\nabla u_{k}|^{p}d\mu \leq \lambda_{1} +o_{k}(1)||u_{k}||_{W_{0}^{1,p}(\Omega)}\eqno{(3.3)}$$
$$\int_{\Omega}|\nabla u_{k}|^{p}d\mu= o_{k}(1)||u_{k}||_{W_{0}^{1,p}(\Omega)}.\eqno{(3.4)}$$
Calculate (3.3)$-\frac{1}{\theta} \times$ (3.4), we have
$$(\frac{1}{p}-\frac{1}{\theta})||u_{m}||^{p}_{W_{0}^{1,p}(\Omega)}\leq M,  \eqno{(3.5)}$$
where $\theta> p$ is a constant. This implies that $\{u_{k}\}$ is bounded in $W_{0}^{1,p}(\Omega)$.
Thus by Lemma \ref{y2}, there exists some $e_{1}\in W_{0}^{1,p}(\Omega)$
 such that $u_{n}\rightarrow e_{1}$ in $ W_{0}^{1,p}(\Omega)$ and
 $$\phi(e_{1})=\int_{\Omega}|\nabla e_{1}|^{p}d\mu\leq\liminf _{n\rightarrow\infty}
 \int_{\Omega}|\nabla u_{n}|^{p}d\mu=\liminf _{n\rightarrow\infty} \phi(u_{n})=\inf (P)\eqno{(3.6)}$$
 and $\int_{\Omega}K(x)|e_{1}|^{p}d\mu =1$.
 It is clear that $e_{1}$ is a solution of the problem ($P$). Moreover, since $|e_{1}|$ is also a solution of the problem ($P$),
 we may assume $e_{1}\geq 0$.

 Since for every $v\in C_{c}^{1}(\Omega)$, we have
 $$\frac{d}{d\varepsilon}|_{\varepsilon=0}\frac{\phi(e_{1}+\varepsilon v)}{\psi(e_{1}+\varepsilon v)}=0. \eqno{(3.7)}$$
 So $e_{1}$ is an eigenfunction of the problem ($P$) corresponding to the principal eigenvalue
 $\lambda_{1}=\int_{\Omega}|\nabla e_{1}|^{p}d\mu .$   ~$\Box$\\

 \indent{\textbf{The Proof of Theorem 2.5.}
 By Theorem \ref{T11}, we have
$$\lambda_{1}=\inf_{u\not \equiv 0,u|_{\partial \Omega}=0} \frac{\int_{\Omega}|\nabla u|^{p}d\mu}{\int_{\Omega}K(x)|u|^{p}d\mu}=\inf_{u\in M}\phi(u). \eqno{(3.8)} $$

Let $ u_{1}$ be an eigenfunction associated to
$\lambda_{1}(K_{1})$. So we have
$$\lambda_{1}(K_{1})= \frac{\int_{\Omega}|\nabla u_{1}|^{p}d\mu}{\int_{\Omega}K(x)|u_{1}|^{p}d\mu}. \eqno{(3.9)}$$
Since $K_{1}<K_{2}$ and $\mu(x)>0$ for all $x\in \Omega$, we have\\
$$\int_{\Omega}K_{1}(x)|u_{1}|^{p}d\mu =\sum_{x\in\Omega}K_{1}(x)|u_{1}(x)|^{p}\mu(x)
<\sum_{x\in\Omega}K_{2}(x)|u_{1}(x)|^{p}\mu(x)=\int_{\Omega}K_{2}(x)|u_{1}|^{p}d\mu.\eqno{(3.10)} $$
 Using $u_{2}=\frac{u_{1}}{(\int_{\Omega}K_{2}(x)|u_{1}|^{p}d\mu)^{1/p}}$
 as an admissible function in (3.8) for
 $\lambda_{1}(K_{2})$, we have
$$
\lambda_{1}(K_{2})\leq\frac{\int_{\Omega}(|\nabla u_{2}|^{p}d\mu}{\int_{\Omega}K_{2}(x)|u_{2}|^{p}d\mu}
=\frac{\int_{\Omega}(|\nabla u_{1}|^{p}d\mu}{\int_{\Omega}K_{2}(x)|u_{1}|^{p}d\mu}
<\frac{\int_{\Omega}(|\nabla u_{1}|^{p}d\mu}{\int_{\Omega}K_{1}(x)|u_{1}|^{p}d\mu}
 =\lambda_{1}(K_{1}). \eqno{(3.11)}
$$
Thus we have $\lambda_{1}(K_{2})<\lambda_{1}(K_{1})$. ~$\Box$\\

 \indent\textbf{The Proof of Theorem 2.6.} Let $u\in W_{0}^{1,p}(\Omega_{1})$ be an eigenfunction associated to $\lambda_{1}(\Omega_{1})$, that is ,
$$\lambda_{1}(\Omega_{1})= \frac{\int_{\Omega_{1}}|\nabla u|^{p}d\mu}{\int_{\Omega_{1}}K(x)|u|^{p}d\mu}.\eqno{(3.12)} $$
Let $\tilde{u}\in W_{0}^{1,p}(\Omega_{2})$ satisfy
\[
\tilde{u}=
\begin{cases}
  u ,& x\in \Omega_{1},\\
  0, &x\in \Omega_{2}\setminus \Omega_{1}.
\end{cases} \eqno{(3.13)}
\]
So we have
$$\int_{\Omega_{2}}K|\tilde{u}|^{p}d\mu=\sum_{x\in\Omega_{1}}K(x)|u(x)|^{p}\mu(x)= \int_{\Omega_{1}}K|u|^{p}d\mu, \eqno{(3.14)}$$
$$\int_{\Omega_{2}}(|\nabla \tilde{u}|^{p}d\mu=\int_{\Omega_{1}}(|\nabla \tilde{u}|^{p}d\mu.\eqno{(3.15)}$$
Combining (3.14) and (3.15), and
taking $\tilde{u}/(\int_{\Omega_{2}}K\tilde{u}^{p}d\mu)^{1/p}$ as an
admissible function for $\lambda_{1}(\Omega_{2})$, we have
$$
\lambda_{1}(\Omega_{2})\leq\frac{\int_{\Omega_{2}}(|\nabla
\tilde{u}|^{p}d\mu}{\int_{\Omega_{2}}K(x)|\tilde{u}|^{p}d\mu}
=\frac{\int_{\Omega_{1}}(|\nabla \tilde{u}|^{p}d\mu}{\int_{\Omega_{1}}K(x)|\tilde{u}|^{p}d\mu}
 =\lambda_{1}(\Omega_{1}). \eqno{(3.16)}
$$
Thus, the proof is complete. ~$\Box$\\

 \indent\textbf{The Proof of Theorem 2.7.}
 Now, we define the functional $J:W_{0}^{1,p}(\Omega)\rightarrow \mathbb{R}$ by
$$
J(u)=\frac{1}{p}\int_{\Omega}|\nabla u|^{p}d\mu-\frac{\lambda}{p}\int_{\Omega}K(x)| u|^{p}d\mu-\int_{\Omega}F(x,u^{+})d\mu,\eqno{(3.17)}
$$
where $u^{+}(x)=\max\{u(x),0\}$.
 By the definition of $\lambda_{1}$, for $u\in W_{0}^{1,p}(\Omega)$ we have
$$
\lambda_{1}\int_{\Omega}K(x)|u|^{p}d\mu\leq\int_{\Omega}|\nabla u|^{p}d\mu. \eqno{(3.18)}
$$
Since $0<\lambda<\lambda_{1}$, we have
$$
J(u)\geq\frac{1}{p}(1-\frac{\lambda}{\lambda_{1}})\int_{\Omega}|\nabla u|^{p}d\mu-\int_{\Omega}F(x,u^{+})d\mu,\eqno{(3.19)}
$$

Indeed, from (H4), there exist $\tau,\delta>0$
such that if $|u|\leq \delta$ we have
$$f(x,u^{+})\leq \tau (u^{+})^{p-1}.\eqno{(3.20)}$$
On the other hand, by (H3), there exist $c,\beta$ such that
$$f(x,u^{+})\leq
cexp(\beta|u|^{\frac{p}{p-1}}),\forall|u|\geq\delta.\eqno{(3.21)}$$ Then we
obtain that, for $q>p$ ,
$$F(x,u^{+})\leq
cexp(\beta|u|^{\frac{p}{p-1}})|u|^{q},\forall|u|\geq\delta.\eqno{(3.22)}$$
Combining (3.20) and (3.22), we obtain that
$$F(x,u^{+})\leq \tau\frac{|u|^{p}}{p}+
c exp(\beta|u|^{\frac{p}{p-1}})|u|^{q}.\eqno{(3.23)}$$
 By the H$\ddot{o}$lder inequality, we
have
$$J(u)\geq
\frac{1}{p}(1-\frac{\lambda}{\lambda_{1}})||u||^{p}_{W_{0}^{1,p}(\Omega)}-\frac{\tau}{p}\int_{\Omega}|u|^{p}d\mu-c(\int_{\Omega}exp(\beta
p|u|^{\frac{p}{p-1}})d\mu)^{\frac{1}{p}}(\int_{\Omega}|u|^{qp'}d\mu)^{\frac{1}{p'}},\eqno{(3.24)}$$
where $\frac{1}{p}+\frac{1}{p'}=1$. By Lemma \ref{moser}, when $||u||_{W_{0}^{1,p}(\Omega)}\leq 1$ for any $u\in W_{0}^{1,p}(\Omega)$
we obtain that
$$\int_{\Omega}exp(\beta p|u|^{\frac{p}{p-1}})d\mu<C|\Omega|.\eqno{(3.25)}$$
By Lemma \ref{y2}, there exists some constant $C$ that depends only on $p$ and $\Omega$
such that
$$ (\int_{\Omega}|u|^{p}d\mu)^{1/p}\leq C (\int_{\Omega}|\nabla u|^{p}d\mu)^{1/p}=C||u||_{W_{0}^{1,p}(\Omega)} \eqno{(3.26)}$$
$$ (\int_{\Omega}|u|^{qp'}d\mu)^{1/p'}\leq C (\int_{\Omega}|\nabla u|^{p}d\mu)^{q/p}=C||u||^{q}_{W_{0}^{1,p}(\Omega)} .\eqno{(3.27)}$$
By (3.24), (3.25), (3.26) and (3.27), we can find some sufficiently small $r>0$ such that if $||u||_{W_{0}^{1,p}(\Omega)}=r$ we have
$$J(u)\geq
\frac{1}{p}(1-\frac{\lambda}{\lambda_{1}}-C\tau)||u||^{p}_{W_{0}^{1,p}(\Omega)}-C||u||^{q}_{W_{0}^{1,p}(\Omega)}. \eqno{(3.28)}$$
By (H4), we set $\tau<\frac{1}{C}(1-\frac{\lambda}{\lambda_{1}})$. Since $q>p>2$, we have
$$\inf_{||u||_{W_{0}^{1,p}(\Omega)}=r} J(u)>0. \eqno{(3.29)}$$
By (H5), there exist two positive constants $c_{1}$ and $c_{2}$ such that
$$F(x,u^{+})\geq c_{1}(u^{+})^{q}-c_{2}. \eqno{(3.30)}$$
Take $u_{0}\in W_{0}^{1,p}(\Omega)$ such that $u_{0}\geq 0$ and $u_{0}\not\equiv0$. For any $t>0$, we have
$$J(tu_{0})\leq\frac{t^{p}}{p}||u_{0}||^{p}_{W_{0}^{1,p}(\Omega)}-\frac{t^{p}}{p}\lambda\int_{\Omega}K(x)|u_{0}|^{p}d\mu  -c_{1}t^{q}\int_{\Omega}u_{0}^{q}d\mu+c_{2}|\Omega|.  \eqno{(3.31)}$$
Since $q>p>2$, we have $J(tu_{0})\rightarrow -\infty$ as $t\rightarrow +\infty$. Hence there exists some $u_{1}\in W_{0}^{1,p}(\Omega)$
satisfying
$$J(u_{1})<0, ~~~~~ ||u_{1}||_{W_{0}^{1,p}(\Omega)}>r. \eqno{(3.32)}  $$

Now we prove that $J(u)$ satisfies the $(PS)_{c}$ condition for any $c\in R$. To see this, we assume
$J(u_{k})\rightarrow c$ and $J'(u_{k})\rightarrow 0$ as $k\rightarrow \infty$, that is
$$\frac{1}{p}\int_{\Omega}|\nabla u_{k}|^{p}d\mu -
\frac{1}{p}\lambda\int_{\Omega}K(x)|u_{k}|^{p}d\mu -\int_{\Omega}F(x,u_{k}^{+})d\mu=c+o_{k}(1) \eqno{(3.33)}$$
$$\int_{\Omega}|\nabla u_{k}|^{p}d\mu-\lambda\int_{\Omega}K(x)|u_{k}|^{p}d\mu-\int_{\Omega}u_{k}f(x,u_{k}^{+})d\mu= o_{k}(1)||u_{k}||_{W_{0}^{1,p}(\Omega)}.\eqno{(3.34)}$$
By the definition of $\lambda_{1}$, (3.33) and (H5), we have
$$
\frac{1}{p}(1-\frac{\lambda}{\lambda_{1}})\int_{\Omega}|\nabla u_{k}|^{p}d\mu\leq \int_{\Omega}F(x,u_{k}^{+})d\mu+c+o_{k}(1)
\leq \frac{1}{q}\int_{\Omega}u_{k}f(x,u_{k}^{+})d\mu+c+o_{k}(1). \eqno{(3.35)}
$$
By (3.34) and (3.35), we get
$$
(\frac{1}{p}-\frac{1}{q})(1-\frac{\lambda}{\lambda_{1}})\int_{\Omega}|\nabla u_{k}|^{p}d\mu
\leq C - \frac{1}{q}o_{k}(1)||u_{k}||_{W_{0}^{1,p}(\Omega)}. \eqno{(3.36)} $$
Since $q>p$, we have $||u_{k}||_{W_{0}^{1,p}(\Omega)}\leq M$.
Thus $\{u_{k}\}$ is bounded in $W_{0}^{1,p}(\Omega)$.
Then the $(PS)_{c}$ condition follows by Lemma \ref{y2}.

Combining (3.29), (3.32) and the obvious fact that $J(0)=0$, we conclude by Lemma \ref{y1} that there
exists a function $u\in W_{0}^{1,p}(\Omega)$ such that
$$J(u)=\inf_{\gamma\in\Gamma}\max_{t\in[0,1]}J(\gamma(t))>0 \eqno{(3.37)}$$
and $J'(u)= 0,$
where $\Gamma=\{\gamma\in C([0,1], W_{0}^{1,p}(\Omega)): \gamma(0)=0,\gamma(1)=u_{1}\} $.
Hence there exists a nontrivial solution $u\in W_{0}^{1,p}(\Omega)$ to the equation
 $$
\left \{
\begin{array}{lcr}
 -\Delta_{p}u= \lambda K(x)|u|^{p-2}u+
  f(x,u),   x \in \Omega^{\circ},\\
  u=0,  x\in\partial \Omega.
 \end{array}
\right.
$$

Testing the above equation by $u^{-}=\min\{u,0\}$ and noting that
$$\Gamma(u^{-},u)=\Gamma(u^{-},u^{-})+\Gamma(u^{-},u^{+})\geq |\nabla u^{-}|^{2}, \eqno{(3.38)}$$
since $0<\lambda<\lambda_{1}$, we have
\begin{align}
\int_{\Omega}|u^{-}|^{p}d\mu-\lambda\int_{\Omega}K(x)|u^{-}|^{p}d\mu&\leq
 -\int_{\Omega}u^{-}\Delta_{p}u d\mu-\lambda\int_{\Omega}K(x)u^{-}|u^{-}|^{p-2}ud\mu  \notag  \\
 &= \int_{\Omega}u^{-} f(x,u^{+})d\mu =0.\notag
\end{align}
This implies that $u^{-}\equiv 0$ and thus $u\geq 0$.~$\Box$\\

\end{document}